\pdfoutput=1
\documentclass[11pt,a4paper]{amsart}
\usepackage[margin=1in]{geometry}
\usepackage[utf8]{inputenc}
\usepackage[T1]{fontenc}
\usepackage{textcase}
\usepackage[dvipsnames]{xcolor}
\usepackage{microtype}
\usepackage{fnpct}

\usepackage{amsmath}
\usepackage{amssymb}
\usepackage[mathscr]{eucal}
\usepackage{tikz-cd}
\renewcommand{\injlim}{\varinjlim}

\usepackage{enumitem}
\usepackage{booktabs}
\usepackage[pdfusetitle,colorlinks]{hyperref}
\hypersetup{bookmarksdepth=2,pdfencoding=unicode,allcolors=MidnightBlue}
\usepackage[capitalise,noabbrev,nosort]{cleveref}

\crefname{equation}{}{}
\creflabelformat{enumi}{(#2#1#3)}
\crefname{enumi}{}{}
\newlist{conenum}{enumerate}{1}
\setlist[conenum,1]{label=(\roman*),ref=\roman*}
\creflabelformat{conenumi}{(#2#1#3)}
\crefname{conenumi}{}{}

\numberwithin{equation}{section}
\theoremstyle{plain}
\newtheorem{Theorem}{Theorem}

\crefname{Theorem}{Theorem}{Theorems}
\newtheorem{conjecture}[equation]{Conjecture}
\newtheorem{corollary}[equation]{Corollary}
\newtheorem{lemma}[equation]{Lemma}
\newtheorem{proposition}[equation]{Proposition}
\newtheorem{theorem}[equation]{Theorem}

\theoremstyle{definition}
\newtheorem{definition}[equation]{Definition}

\newtheorem{example}[equation]{Example}

\theoremstyle{remark}
\newtheorem{remark}[equation]{Remark}

\let\oldSS\SS\let\SS\relax

\newcommand{\SS}{\mathbf{S}}

\newcommand{\Ste}{\textnormal{Ste}}

\newcommand{\ic}{\textnormal{ic}}

\newcommand{\lex}[1]{#1\textnormal{-lex}}
\newcommand{\pr}{\textnormal{pr}}

\newcommand{\st}{\textnormal{st}}

\newcommand{\Alg}{\operatorname{Alg}}
\newcommand{\CAlg}{\operatorname{CAlg}}

\newcommand{\End}{\operatorname{End}}
\newcommand{\Fun}{\operatorname{Fun}}
\newcommand{\Hom}{\operatorname{Hom}}
\newcommand{\IND}{\operatorname{IND}}
\newcommand{\Ind}{\operatorname{Ind}}
\newcommand{\LMod}{\operatorname{LMod}}
\newcommand{\Map}{\operatorname{Map}}
\newcommand{\Mod}{\operatorname{Mod}}

\newcommand{\PShv}{\operatorname{PShv}}

\newcommand{\cofib}{\operatorname{cofib}}

\newcommand{\op}{\operatorname{op}}

\newcommand{\Kappa}{\mathfrak{K}}
\newcommand{\X}{\mathord{-}}
\newcommand{\unit}{\mathbf{1}}

\newcommand{\Cat}[1]{\mathsf{#1}}
\newcommand{\cat}[1]{\mathscr{#1}}

\title[Higher presentable categories]{Higher presentable categories and limits}
\author{Ko Aoki}
\address{Max Planck Institute for Mathematics,
  Vivatsgasse 7, 53111 Bonn, Germany
}
\email{aoki@mpim-bonn.mpg.de}
\date{\today}

\begin{document}

\begin{abstract}
  Stefanich generalized
  the notion of (locally) presentable \((\infty,1)\)-category
  to the notion of presentable \((\infty,n)\)-category.
  We give a new description
  based on the new notion of
  \(\kappa\)-compactly generated \((\infty,n)\)-category,
  which avoids universe enlargement.

  Using the new definition,
  we prove the underlying functor
  of a morphism between presentable \((\infty,2)\)-categories
  has a right adjoint.
  In particular,
  any presentable \((\infty,2)\)-category has limits.
  We also prove that this fails drastically
  when we go higher:
  The unit presentable \((\infty,3)\)-category,
  i.e., the category of presentable \((\infty,2)\)-categories,
  does not have limits.
  This settles Stefanich’s conjecture in the negative.
\end{abstract}

\maketitle
\setcounter{tocdepth}{1}
\tableofcontents

\section{Introduction}\label{s:intro}

In this paper,
“category” means an \((\infty,1)\)-category,
and
“\(n\)-category” means an \((\infty,n)\)-category.

As observed in, e.g.,
Giraud’s characterization in topos theory,
the notion of presentability,
or simply the sense of “being generated by small objects”
of a large category,
plays a fundamental role.
We refer the reader to~\cite[Section~5.5]{LurieHTT}
for a foundation of the theory of presentable categories.
One powerful tool is the adjoint functor theorem,
which asserts the underlying functor
of any morphism between presentable categories
has a right adjoint.
This in particular implies the existence of limits;
given that
presentability is defined in terms of colimits,
this is nontrivial:

\begin{example}\label{xqud1s}
  The category freely generated by small colimits
  on a proper class
  does not have small limits.
  Note that it is still locally small.
\end{example}

It would be useful to have
a similar theory in higher category theory.
We seek a notion of presentable \(n\)-category\footnote{The term “presentable \(n\)-category”
  is sometimes used to mean
  a presentable \((n-1)\Cat{Cat}\)-enriched category.
  This is a different notion unless \(n=1\).
} for \(n\geq0\).
Naively, the category of presentable \(2\)-categories~\(2\Cat{Pr}\)
should be the category
of “presentable \(\Cat{Pr}\)-enriched categories,”
and we should similarly get \(n\Cat{Pr}\) by repeating this process.
However,
since
the category of presentable \(1\)-categories~\(\Cat{Pr}\)
itself is not presentable,
we run into a size issue.
To avoid this,
Stefanich~\cite{Stefanich} devised an ingenious solution,
which we recall in \cref{ss:orig}.
One downside of his approach
is that it uses the enlargement of the universe.
Pedantically, the existence of such an enlargement
is a strong axiom whose relative consistency over Zermelo–Fraenkel set theory
is not provable,
but this is widely accepted in subjects like algebraic geometry.
Still,
even assuming that it is possible,
two different enlargement of the universe
could yield two a~priori different notions of a presentable \(n\)-category.
Moreover, from his presentation,
certain problems are difficult to study.
We take advantage of our new approach
to disprove his conjecture; see below.

This paper is about the following new definition;
see \cref{ss:new} for a detailed explanation:

\begin{definition}\label{x2r4as}
  For a regular cardinal~\(\kappa\),
  let \(\Cat{Pr}^{\kappa}\)
  be the category
  of \(\kappa\)-compactly generated categories
  and functors preserving colimits and \(\kappa\)-compact objects.
  We define \(0\Cat{Pr}^{\kappa}\)
  to be the category of animas (aka homotopy types)~\(\Cat{Ani}\)
  and
  \((n+1)\Cat{Pr}^{\kappa}=\Mod_{n\Cat{Pr}^{\kappa}}\Cat{Pr}^{\kappa}\)
  for \(n\geq0\).
  We call it the category
  of \emph{\(\kappa\)-compactly generated \(n\)-categories}.
  For regular cardinals \(\kappa\leq\lambda\),
  we get a canonical functor \(n\Cat{Pr}^{\kappa}\to n\Cat{Pr}^{\lambda}\).
  We consider
  \begin{equation*}
    n\Cat{Pr}=\injlim_{\kappa}n\Cat{Pr}^{\kappa}
  \end{equation*}
  and call it the category of \emph{presentable \(n\)-categories}.
\end{definition}

This recovers the original definition:

\begin{Theorem}\label{comparison}
  With any enlargement of the universe,
  \cref{x2r4as} coincides with Stefanich’s definition.
  In particular,
  his definition
  does not depend on the enlargement.
\end{Theorem}

We then go back to the problem of limits.
This was widely open; Stefanich posed the following
in~\cite[Conjecture~5.5.1]{Stefanich}:

\begin{conjecture}[Stefanich]\label{xiizpr}
  The \(1\)-category~\(n\Cat{Pr}\)
  has (small) limits for~\(n\geq0\).
\end{conjecture}

\begin{remark}\label{xufuyz}
  In Stefanich’s theory,
  the limit can be computed in \(\Mod_{(n-1)\Cat{Pr}}(\Cat{REX}^{\Kappa})\)
  (see \cref{ss:orig} for our notation),
  and the question
  is whether it belongs to the full subcategory \(n\Cat{Pr}\).
\end{remark}

\begin{remark}\label{xd12jz}
In \cref{xiizpr},
  we only consider \(1\)-categorical limits,
  since it is straightforward to deduce
  the existence of \((n+1)\)-categorical limits
  from this fact.
\end{remark}

\begin{example}\label{xgxs7p}
  When \(n=0\), \cref{xiizpr} is the fact that \(\Cat{Ani}\) has limits.
\end{example}

\begin{example}\label{x2mzo5}
  When \(n=1\), \cref{xiizpr} is~\cite[Proposition~5.5.3.13]{LurieHTT};
  we review this in \cref{ss:lim_pr}.
  In the classical setting, this is due to Bird~\cite{Bird84}.
\end{example}

Therefore,
the first open case was \(n=2\).
We prove the following:

\begin{Theorem}\label{false}
  \Cref{xiizpr} is false for \(n=2\).
\end{Theorem}

\begin{remark}\label{xgspx6}
  Note that
  \cref{false} implies that
  the adjoint functor theorem
  fails for presentable \(3\)-categories.
  Hence,
  the original notation~\(3\Cat{Pr}^{\textnormal{L}}\)
  may be misleading.
\end{remark}

Moreover,
we obtain something negative
about~\(\Hom\),
which answers the question in~\cite[Remark~1.1.3]{Stefanich}.
We first explain what \(\Hom\) is:

\begin{definition}\label{xvgpfp}
  Let \(C\) and~\(C'\)
  be objects of a presentable \(n\)-category~\(\cat{C}\)
  for \(n\geq1\).
  We say that \emph{\(\Hom_{\cat{C}}(C,C')\) exists presentably}
  if there is an object satisfying
  \begin{equation*}
    \Map_{(n-1)\Cat{Pr}}(\cat{E},\Hom_{\cat{C}}(C,C'))
    \simeq
    \Map_{\cat{C}}(\cat{E}\otimes C,C')
  \end{equation*}
  for any \(\cat{E}\in(n-1)\Cat{Pr}\).
\end{definition}

\begin{remark}\label{xc7puq}
  As in \cref{xufuyz},
  in Stefanich’s theory,
  \(\Hom\) in a presentable \(n\)-category
  for \(n\geq2\)
  always exists as an object
  of \(\Mod_{(n-2)\Cat{Pr}}(\Cat{REX}^{\Kappa})\).
  We are asking in \cref{xvgpfp}
  whether it belongs to the full subcategory \((n-1)\Cat{Pr}\).
\end{remark}

\begin{example}\label{xcv88j}
  In \cref{xvgpfp},
  when \(n=1\),
  \(\Hom\) is just \(\Map\).
\end{example}

\begin{example}\label{x0a017}
  In \cref{xvgpfp},
  when \(n=2\),
  \(\Hom\) always exists presentably.
  This follows, e.g., from \cref{adjoint} below.
\end{example}

\begin{Theorem}\label{omega}
  There is a stable presentably symmetric monoidal \(3\)- or \(4\)-category
  whose \(\End(\unit)\) does not exist presentably.
\end{Theorem}

\begin{remark}\label{xayu6y}
  \Cref{omega} is not merely a pathology
  of certain presentable \(n\)-categories.
  This phenomenon of iteratively taking \(\Hom\)
  producing something nonpresentable
  is already reflected in~\(3\Cat{Pr}\) or~\(4\Cat{Pr}\)
  since we have \(\cat{C}\simeq\Hom_{n\Cat{Pr}}(\unit,\cat{C})\)
  for any \(\cat{C}\in n\Cat{Pr}\).
\end{remark}

Nevertheless,
we can still prove something useful
below that level:

\begin{Theorem}\label{adjoint}
  The underlying functor
  of any morphism between
  presentable \(2\)-categories
  has a right adjoint.
  In particular,
  any presentable \(2\)-category has limits.
\end{Theorem}

We conclude this introduction with a few remarks:

\begin{remark}\label{xyvdh1}
  One way to avoid this problem
  is simply to fix~\(\kappa\) (e.g., \(\kappa=\aleph_{1}\))
  and work with \(\kappa\)-compactly generated \(n\)-categories.
  With this,
  the pathology as in \cref{omega}
  does not arise,
  and we can still consider \(\kappa\)-compactly generated categorical spectra;
  see \cref{prsp}.
\end{remark}

\begin{remark}\label{x9h82c}
  The negative results in this paper
  demonstrate that there are fewer adjoints
  in higher category theory.
  In~\cite{n-rig-1},
  we introduce the notion of higher rigidity,
  which ensures the existence of many adjoints.
  In practice,
  many presentable \(n\)-categories appearing in nature
  have certain rigidity.
\end{remark}

\subsection*{Organization}\label{ss:outline}

In \cref{s:pr},
we explain the new definition
and prove \cref{comparison}.
In \cref{s:ni},
we prove some positive results
including \cref{adjoint}.
In \cref{s:san},
we construct counterexamples
showing \cref{false,omega}.

\subsection*{Acknowledgments}\label{ss:ack}

I thank Tim Campion, Adam Dauser, Peter Scholze, and Germán Stefanich
for helpful discussions related to this work.
I thank David Reutter for inviting me to present these results
at the University of Hamburg in December 2024.
I thank Peter Scholze for useful comments on a draft of this paper.
I thank the Max Planck Institute for Mathematics.

\subsection*{Conventions}\label{ss:conv}

Regular cardinals are always assumed to be infinite.
For a category~\(\cat{C}\)
and a regular cardinal~\(\kappa\),
we write \(\cat{C}_{\kappa}\) for
its full subcategory spanned by \(\kappa\)-compact objects.

\section{Presentable \texorpdfstring{\(n\)}{n}-categories}\label{s:pr}

We explain the new definition
in \cref{ss:new}.
We recall Stefanich’s original definition
and compare it with ours
in \cref{ss:orig}.

\subsection{Our definition}\label{ss:new}

We first observe that \cref{x2r4as}
makes sense.
We must inductively show that
\(n\Cat{Pr}^{\kappa}\) is in \(\CAlg(\Cat{Pr}^{\kappa})\).
We first note the following:

\begin{lemma}\label{x6wlid}
  Let \(\kappa\) be a regular cardinal
  and \(\cat{C}\) an object of \(\Alg(\Cat{Pr}^{\kappa})\).
  Then for any \(A\in\Alg(\cat{C})\),
  the category \(\LMod_{A}(\cat{C})\)
  is \(\kappa\)-compactly generated.
\end{lemma}

\begin{proof}
  We consider the forgetful functor \(G\colon\LMod_{A}(\cat{C})\to\cat{C}\)
  and its left adjoint~\(F\).
  Then
  \(F\) preserves \(\kappa\)-compact objects,
  since \(G\) preserves colimits
  by~\cite[Corollary~4.2.3.7]{LurieHA}.
  We can see that the image of~\(F\) generates
  \(\LMod_{A}(\cat{C})\)
  by considering the bar construction.
  Therefore, the claim follows.
\end{proof}

\begin{lemma}\label{xm1vro}
  Let \(\kappa\) be a regular cardinal
  and \(\cat{C}\) an object of \(\CAlg(\Cat{Pr}^{\kappa})\).
  Then for any \(A\in\CAlg(\cat{C})\),
  the category \(\Mod_{A}(\cat{C})\)
  is in \(\CAlg(\Cat{Pr}^{\kappa})\).
\end{lemma}

\begin{proof}
  This follows from the proof of \cref{x6wlid},
  noting that \(F\) is symmetric monoidal.
\end{proof}

Given this,
we are reduced to showing the following statement:

\begin{proposition}\label{xuizju}
  For a regular cardinal~\(\kappa\),
  the category
  \(\Cat{Pr}^{\kappa}\)
  is an object of \(\CAlg(\Cat{Pr}^{\kappa})\).
\end{proposition}

\begin{proof}
  Let \(\Cat{Rex}^{\kappa}\)
  be the category of
  categories with \(\kappa\)-small colimits
  and functors that preserve them
  and \(\Cat{Rex}^{\kappa}_{\ic}\)
  its full subcategory
  spanned by idempotent complete categories.
  When \(\kappa>\aleph_{0}\),
  this inclusion is an equivalence.
  We recall that
  \({\Ind_{\kappa}}\colon\Cat{Rex}^{\kappa}_{\ic}\to\Cat{Pr}^{\kappa}\)
  induces an equivalence;
  see~\cite[Propositions~5.5.7.8 and~5.5.7.10]{LurieHTT}.
  Hence, it suffices to show that
  \(\Cat{Rex}^{\kappa}_{\ic}\)
  is an object of \(\CAlg(\Cat{Pr}^{\kappa})\).

  It is presentable by~\cite[Lemma~4.8.4.2]{LurieHA}.
  We wish to prove that it is \(\kappa\)-compactly generated.
  We consider the forgetful functor
  \(G\colon\Cat{Rex}^{\kappa}_{\ic}\to\Cat{Cat}\).
  Its left adjoint~\(F\) generates \(\Cat{Rex}^{\kappa}_{\ic}\).
  When \(\kappa>\aleph_{0}\),
  the desired result directly follows
  from \cref{xgvgzm} below.
  When \(\kappa=\aleph_{0}\),
  the desired result follows
  from combining it with~\cite[Corollary~4.4.5.21]{LurieHTT}.
  We then have to check that the tensor product operations
  preserve \(\kappa\)-compactness.
  This follows from the fact that \(F\) is symmetric monoidal.
\end{proof}

\begin{lemma}\label{xualzb}
  The category of categories~\(\Cat{Cat}\),
  equipped with the cartesian symmetric monoidal structure,
  is an object of \(\CAlg(\Cat{Pr}^{\aleph_{0}})\).
  Moreover, the following hold:
  \begin{enumerate}
    \item\label{i:sm_po}
      The walking span~\(\Lambda_{0}^{2}\) is compact.
    \item\label{i:sm_dis}
      Any \(\kappa\)-small set is \(\kappa\)-compact
      for a regular cardinal~\(\kappa\).
  \end{enumerate}
\end{lemma}

\begin{proof}
  Recall that~\(\Cat{Cat}\)
  can be identified with the category of complete Segal animas.
  There is a fully faithful functor
  \(
    G\colon\Cat{Cat}\to\Fun(\Cat{\Delta}^{\op},\Cat{Ani})
  \)
  given concretely by
  \(\cat{C}\mapsto([n]\mapsto\Fun([n],\cat{C})^{\simeq})\),
  which preserves filtered colimits.
  Hence its left adjoint~\(F\) preserves compact objects.
  It follows that \(\Cat{Cat}\) is compactly generated
  by compact objects~\([n]\) for \(n\geq 0\).
  Since \([0]\) is a retract of~\([1]\)
  and \([n+1]=[n]\amalg_{[0]}[1]\) for \(n\geq0\),
  it is generated by~\([1]\).
  To see that \(\Cat{Cat}\) lies in \(\CAlg(\Cat{Pr}^{\aleph_{0}})\),
  it suffices to show that \([1]\times[1]\) is compact,
  which holds since it can be presented as \([2]\amalg_{[1]}[2]\).

  Statement~\cref{i:sm_po} follows
  from the presentation \(\Lambda^{2}_{0}=[1]\amalg_{[0]}[1]\),
  while~\cref{i:sm_dis} is immediate,
  as it is a \(\kappa\)-small coproduct of~\([0]\).
\end{proof}

\begin{lemma}\label{xgvgzm}
  For a regular cardinal~\(\kappa\),
  the forgetful functor
  \(G\colon\Cat{Rex}^{\kappa}\to\Cat{Cat}\)
  preserves \(\kappa\)-filtered colimits.
\end{lemma}

\begin{proof}
  Let \(P\) be a \(\kappa\)-directed poset.
  Consider a colimit diagram \(\cat{C}\colon P^{\triangleright}\to\Cat{Cat}\)
  such that \(\cat{C}\rvert_{P}\) takes values to~\(\Cat{Rex}^{\kappa}\).
  We need to see that
  this factors through~\(\Cat{Rex}^{\kappa}\).
  We can show this concretely,
  but we give a more conceptual proof here.
  Let \(\cat{K}\) be either of the two categories in \cref{xualzb}.
  By \cref{xualzb},
  \(\Fun(\cat{K},\cat{C}(\X))\) is also a colimit diagram.
  We have the diagonal \(R\colon\cat{C}\to\Fun(\cat{K},\cat{C})\).
  By assumption,
  this satisfies
  the hypothesis of \cref{xh9qmw} below
  with “right” replaced by “left.”
  Hence,
  we see that
  \(\cat{C}(\infty)\) admits colimits indexed by~\(\cat{K}\)
  and
  \(\cat{C}(p)\to\cat{C}(\infty)\)
  preserves them for any \(p\in P\).
  Therefore, the claim follows.
\end{proof}

\begin{lemma}\label{xh9qmw}
  Let \(\cat{K}\) be an \(\infty\)-category
  and \(F(\X)\colon\cat{C}(\X)\to\cat{D}(\X)\)
  be a natural transformation
  between objects in \(\Fun(\cat{K}^{\triangleright},\Cat{Cat})\).
  Suppose that any morphism \(K\to L\) in~\(\cat{K}\),
  the square
  \begin{equation*}
    \begin{tikzcd}
      \cat{C}(K)\ar[r]\ar[d]&
      \cat{D}(K)\ar[d]\\
      \cat{C}(L)\ar[r]&
      \cat{D}(L)
    \end{tikzcd}
  \end{equation*}
  is right adjointable.
  When \(\cat{C}\) and \(\cat{D}\) are colimit diagrams,
  we have the same adjointability property
  for any morphism \(K\to\infty\) in~\(\cat{K}^{\triangleright}\).
\end{lemma}

\begin{proof}
  We use \(2\)-category theory.
  By using~\cite[Theorem~4.6]{Haugseng21},
  we see that
  in this situation,
  \(F(\X)\) admits a left adjoint
  in \(\Fun(\cat{K},\Cat{Cat})\).
  The desired result follows
  from the \(2\)-categorical functoriality
  of the left Kan extension functor
  \(\Fun(\cat{K},\Cat{Cat})
  \to\Fun(\cat{K}^{\triangleright},\Cat{Cat})\).
\end{proof}

We record the following
as an application of the argument above:

\begin{proposition}\label{monadic}
For a regular cardinal~\(\kappa\),
  the forgetful functor
  \(G\colon\Cat{Rex}^{\kappa}\to\Cat{Cat}\)
  is monadic.
\end{proposition}

\begin{proof}
  We use Lurie’s version of the Beck monadicity theorem.
  Since \(G\) is conservative,
  it suffices to show that
  \(G\) preserves \(G\)-split geometric realizations.
  Consider such a simplicial diagram in \(\Cat{Rex}^{\kappa}\).
  We take its (split) geometric realization in \(\Cat{Cat}\)
  to obtain an augmented simplicial object \(\cat{C}_{\bullet}\).
  It suffices to show that this lifts to \(\Cat{Rex}^{\kappa}\).
  Since \(\Fun(\cat{K},\cat{C}_{\bullet})\)
  is also a split simplicial object in~\(\Cat{Cat}\)
  and hence a colimit diagram,
  the desired result follows the same argument
  as in \cref{xgvgzm}.
\end{proof}

We conclude this subsection with a few remarks:

\begin{remark}\label{xvm1ia}
  We only need the fact that \(\Cat{Pr}^{\kappa}\) is presentable
  to obtain our definition of \(n\Cat{Pr}\):
  For a sequence of regular cardinals
  \(\kappa_{1}\leq\dotsb\leq\kappa_{n}\),
  we can define
  \(n\Cat{Pr}^{\kappa_{1},\dotsc,\kappa_{n}}\)
  inductively as
  \(0\Cat{Pr}=\Cat{Ani}\) and
  \begin{equation*}
    (n+1)\Cat{Pr}^{\kappa_{1},\dotsc,\kappa_{n+1}}
    =\Mod_{n\Cat{Pr}^{\kappa_{1},\dotsc,\kappa_{n}}}(\Cat{Pr}^{\kappa_{n+1}})
  \end{equation*}
  and we can take its colimit.
\end{remark}

\begin{remark}\label{xhq2j7}
  Our proof of \cref{xuizju}
  shows that
  for regular cardinals~\(\kappa\leq\lambda\leq\mu\),
  the functor
  \(n\Cat{Pr}^{\kappa}\to n\Cat{Pr}^{\lambda}\)
  preserves \(\mu\)-compact objects.
  Using this,
  we can also write~\(n\Cat{Pr}\)
  as \(\injlim_{\kappa}(n\Cat{Pr}^{\kappa})_{\kappa}\).
\end{remark}

\begin{remark}\label{prsp}
  For \(n\geq0\),
  we have a fully faithful functor
  \begin{equation*}
    \Mod_{(\X)}(n\Cat{Pr}^{\kappa})\colon
    \CAlg(n\Cat{Pr}^{\kappa})\to\CAlg((n+1)\Cat{Pr}^{\kappa}),
  \end{equation*}
  which is a morphism in \(\CAlg(\Cat{Pr}^{\kappa})\)
  by \cref{x8fchl} below.
  We then take the colimit
  \begin{equation*}
    \Cat{PrSp}^{\kappa}=\injlim_{n}\CAlg(n\Cat{Pr}^{\kappa})
  \end{equation*}
  in \(\CAlg(\Cat{Pr}^{\kappa})\)
  and call it the category
  of \emph{\(\kappa\)-compactly generated categorical spectra}
  (see references in~\cite[Remark~3.3.6]{Masuda24}
  for the origin of the notion of categorical spectra).
  Concretely,
  an object of \(\Cat{PrSp}^{\kappa}\)
  is a sequence \((\cat{C}_{n})_{n\geq0}
  \in\prod_{n\geq0}\CAlg(n\Cat{Pr}^{\kappa})\)
  equipped
  with equivalences \(\cat{C}_{n}\simeq\End_{\cat{C}_{n+1}}(\unit)\).

  Furthermore,
  we can define
  \begin{equation*}
    \Cat{PrSp}=\injlim_{\kappa}\Cat{PrSp}^{\kappa}
  \end{equation*}
  to be the category of \emph{presentable categorical spectra}.
  However,
  this notion is subtler,
  since an object is \emph{not} a sequence
  \((\cat{C}_{n})_{n\geq0}\in\prod_{n\geq0}\CAlg(n\Cat{Pr})\)
  with equivalences \(\cat{C}_{n}\simeq\End_{\cat{C}_{n+1}}(\unit)\)
  as \cref{omega} shows.
\end{remark}

\begin{lemma}\label{x8fchl}
For \(\cat{C}\in\CAlg(\Cat{Pr}^{\kappa})\),
  the functor
  \begin{equation*}
    \Mod_{(\X)}(\cat{C})\colon
    \CAlg(\cat{C})
    \to
    \CAlg(\Mod_{\cat{C}}(\Cat{Pr}^{\kappa}))
  \end{equation*}
  is fully faithful
  and determines a morphism
  in \(\CAlg(\Cat{Pr}^{\kappa})\).
\end{lemma}

\begin{proof}
  By~\cite[Corollary~4.8.5.21]{LurieHTT},
  this is fully faithful
  and preserves colimits.
  Hence it suffices to show that
  its right adjoint \(\cat{D}\mapsto\End_{\cat{D}}(\unit)\)
  preserves \(\kappa\)-filtered colimits.
  Consider a \(\kappa\)-filtered diagram
  \((\cat{D}_{i})_{i}\) in the target.
  For any \(C\in\cat{C}_{\kappa}\),
  we have
  \begin{equation*}
      \Map_{\cat{C}}\biggl(C,\injlim_{i}\End_{\cat{D}_{i}}(\unit)\biggr)
      \simeq
      \injlim_{i}\Map_{\cat{C}}(C,\End_{\cat{D}_{i}}(\unit))
      \simeq
      \injlim_{i}\Map_{\cat{D}_{i}}(C\otimes\unit,\unit).
  \end{equation*}
  Since \(C\otimes\unit\) and~\(\unit\)
  are \(\kappa\)-compact in \(\cat{D}_{i}\),
  by \cref{xgvgzm},
  we can continue the above equivalence as
  \begin{equation*}
    \injlim_{i}\Map_{\cat{D}_{i}}(C\otimes\unit,\unit)
    \simeq
    \Map_{\injlim_{i}\cat{D}_{i}}(C\otimes\unit,\unit)
    \simeq
    \Map_{\cat{C}}\biggl(C,\End_{\injlim_{i}\cat{D}_{i}}(\unit)\biggr),
  \end{equation*}
  and hence the claim follows.
\end{proof}

\subsection{More on compactness}\label{ss:cpt_more}

We here point out that \cref{xuizju}
is the best possible bound:

\begin{proposition}\label{xwabf7}
  Let \(\kappa<\lambda\) be regular cardinals.
  An object
  \(\cat{C}\in\Cat{Pr}^{\lambda}\)
  is \(\kappa\)-compact
  if and only if
  \(\cat{C}=0\).
\end{proposition}

\begin{proof}
  Let \(I\) be a set of cardinality~\(\kappa\).
  We consider the coproduct \(\bigoplus_{i\in I}\cat{C}\)
  in~\(\Cat{Pr}\).
  By \cref{prod},
  it is also a product in \(\Cat{Pr}^{\lambda}\)
  and hence
  we obtain a diagonal morphism
  \(\cat{C}\to\bigoplus_{i\in I}\cat{C}\)
  in \(\Cat{Pr}^{\lambda}\).
  Since \(\cat{C}\) is \(\kappa\)-compact,
  this should factor
  through
  \(\bigoplus_{i\in I_{0}}\cat{C}\),
  where \(I_{0}\subset I\)
  is a subset with cardinality~\(<\kappa\).
  This is possible
  only when \(\cat{C}=0\).
\end{proof}

\subsection{Comparison with the original definition}\label{ss:orig}

We prove \cref{comparison} here
as \cref{xmhmpq}.
We first need to recall Stefanich’s definition.

In this section,
we take an enlargement of the universe
and write \(\Kappa=\lvert\Cat{Set}\rvert\).
As an ordinal, this is just \(\Cat{Ord}\),
the category of small ordinals.
We write \(\Cat{SET}\) for the category of sets
in the enlarged universe
and construct familiar categories from it;
we use capital letters for them.

\begin{example}\label{xz6q0c}
  We have \(\Cat{Ani}\simeq\Cat{ANI}_{\Kappa}\).
  In particular,
  we have \(\Cat{ANI}\simeq\IND_{\Kappa}(\Cat{Ani})\).
\end{example}

\begin{example}\label{xq40hr}
  For a small regular cardinal~\(\kappa\),
  we have
  \((\Cat{REX}^{\kappa})_{\Kappa}\simeq\Cat{Rex}^{\kappa}\).
  By \cref{xuizju},
  we have \(\Cat{REX}^{\kappa}\simeq\IND_{\Kappa}(\Cat{Rex}^{\kappa})\).
\end{example}

The point is that
\(\Cat{PR}\),
the category of presentable categories
in the enlarged universe,
is not equivalent to \(\IND_{\Kappa}(\Cat{Pr})\),
since \(\Cat{PR}_{\Kappa}\)
only contains~\(0\) by \cref{xwabf7}.
Instead, we have the following,
which was observed in~\cite[Proposition~5.1.4]{Stefanich}:

\begin{proposition}[Stefanich]\label{xjl7pp}
  The category
  \(\Cat{REX}^{\Kappa}\) is \(\Kappa\)-compactly generated and
  \(\Cat{Pr}\) is equivalent to
  the full subcategory spanned by \(\Kappa\)-compact objects.
  In particular,
  we have \(\Cat{REX}^{\Kappa}\simeq\IND_{\Kappa}(\Cat{Pr})\).
\end{proposition}

\begin{remark}\label{xm4yc7}
  Since we use \cref{xjl7pp} below,
  we sketch an alternative proof.
  First,
  note that the colimit \(\Cat{Pr}=\injlim_{\kappa<\Kappa}\Cat{Pr}^{\kappa}\)
  can be taken in~\(\Cat{REX}^{\Kappa}\)
  instead of in \(\Cat{CAT}\) by \cref{xgvgzm}.
  Hence, in light of \cref{xq40hr},
  it suffices to show that
  \(\injlim_{\kappa<\Kappa}\Cat{REX}^{\kappa}\to\Cat{REX}^{\Kappa}\)
  is an equivalence in~\(\Cat{PR}\) (and hence in \(\Cat{PR}^{\Kappa}\)).
  We can then observe this directly;
  the key point is that
  having (or preserving) \(\Kappa\)-small colimits
  is equivalent to having (or preserving) \(\kappa\)-small colimits
  for all \(\kappa<\Kappa\).
\end{remark}

Based on this observation,
Stefanich made the following definition:

\begin{definition}[Stefanich]\label{xi91go}
  We consider the lax symmetric monoidal functor
  \begin{equation*}
    \LMod_{\pr}\colon\Alg(\Cat{REX}^{\Kappa})\to\Cat{REX}^{\Kappa}
  \end{equation*}
  that carries \(\cat{C}\)
  to the \(1\)-category of \(\Kappa\)-compact left \(\cat{C}\)-modules.
  We write \(\Mod_{\pr}\)
  for what we obtain by applying \(\CAlg\) to this,
  which is an endofunctor of \(\CAlg(\Cat{REX}^{\Kappa})\).
  We define \(n\Cat{Pr}^{\Ste}\) to be
  \((\Mod_{\pr})^{n}(\Cat{Ani})\).
\end{definition}

We prove that this coincides with our definition:

\begin{theorem}\label{xmhmpq}
  There is a canonical equivalence
  \(n\Cat{Pr}\to n\Cat{Pr}^{\Ste}\).
  In particular,
  \(n\Cat{Pr}^{\Ste}\) does not depend on the enlargement.
\end{theorem}

\begin{lemma}\label{xucgor}
  For a small regular cardinal~\(\kappa\)
  and \(\cat{C}\in\Alg(\Cat{Rex}^{\kappa})\),
  the canonical functor
  \begin{equation*}
    \IND_{\Kappa}(\LMod_{\cat{C}}(\Cat{Rex}^{\kappa}))
    \to\LMod_{\cat{C}}(\Cat{REX}^{\kappa})
  \end{equation*}
  is an equivalence.
\end{lemma}

\begin{proof}
  By applying \cref{x6wlid}
  to \cref{xq40hr},
  we see that
  the target is \(\Kappa\)-compactly generated.
  Moreover,
  by the proof of \cref{x6wlid},
  we see that the full subcategory
  of \(\Kappa\)-compact objects
  is generated by \(\cat{C}\otimes\cat{D}\),
  where \(\cat{D}\) is in~\(\Cat{Rex}^{\kappa}\).
  This is exactly \(\LMod_{\cat{C}}(\Cat{Rex}^{\kappa})\).
\end{proof}

\begin{lemma}\label{xfxio0}
  Suppose that we have a diagram
  \((\cat{C}(\kappa))_{\kappa}\) in \(\Alg(\Cat{Pr})\)
  indexed by small regular cardinals
  (or a cofinal family of them)
  such that \(\cat{C}(\kappa)\) belongs to \(\Alg(\Cat{Pr}^{\kappa})\).
  We write \(\cat{C}\) for the colimit taken in \(\Cat{REX}^{\Kappa}\).
  Then we have a canonical equivalence
  \begin{equation*}
    \injlim_{\kappa<\Kappa}\LMod_{\cat{C}(\kappa)}(\Cat{Pr}^{\kappa})
    \simeq
    (\LMod_{\cat{C}}(\Cat{REX}^{\Kappa}))_{\Kappa}
  \end{equation*}
  in \(\Cat{REX}^{\Kappa}\)
  (or equivalently, in \(\Cat{CAT}\),
  by \cref{xgvgzm}).
\end{lemma}

\begin{proof}
  By \cref{xucgor},
  it suffices to show that the composite
  \begin{equation*}
    \injlim_{\kappa<\Kappa}\LMod_{\cat{C}(\kappa)}(\Cat{REX}^{\kappa})
    \to\injlim_{\kappa<\Kappa}\LMod_{\cat{C}(\kappa)}(\Cat{REX}^{\Kappa})
    \to\LMod_{\cat{C}}(\Cat{REX}^{\Kappa})
  \end{equation*}
  is an equivalence in \(\Cat{PR}\).
  Therefore, it suffices to observe
  that
  \(\injlim_{\kappa<\Kappa}\Cat{REX}^{\kappa}\to\Cat{REX}^{\Kappa}\)
  is an equivalence in~\(\Cat{PR}\),
  but by passing to \(\Kappa\)-compact objects,
  this follows from \cref{xjl7pp}.
\end{proof}

\begin{proof}[Proof of \cref{xmhmpq}]
  We proceed by induction.
  When \(n=0\), this is tautological.
  We assume \(n\geq1\).
  By applying \cref{xfxio0} to \(((n-1)\Cat{Pr}^{\kappa})_{\kappa}\),
  we see that
  \begin{equation*}
    \injlim_{\kappa<\Kappa}(\Mod_{(n-1)\Cat{Pr}^{\kappa}}(\Cat{REX}^{\kappa}))_{\Kappa}
    \to(\Mod_{(n-1)\Cat{Pr}}(\Cat{REX}^{\Kappa}))_{\Kappa}
  \end{equation*}
  is an equivalence in \(\Cat{CAT}\).
  The source is equivalent to \(n\Cat{Pr}=\injlim_{\kappa<\Kappa}n\Cat{Pr}^{\kappa}\)
  by \cref{xucgor}.
  The target is equivalent to \(\Mod_{\pr}((n-1)\Cat{Pr})\) by definition,
  and to \(n\Cat{Pr}^{\Ste}\) by the inductive hypothesis.
\end{proof}

\section{Limits in presentable \texorpdfstring{\(2\)}{2}-categories}\label{s:ni}

In this section,
we study limits in presentable \(2\)-categories.
We show the existence of limits in \cref{ss:fil}
and prove \cref{adjoint} in \cref{ss:aft} using a similar idea.
Our argument relies on enriched category theory,
which we study in \cref{ss:ecg}.

\subsection{A review of limits of presentable categories}\label{ss:lim_pr}

We recall the existence
of limits in~\(\Cat{Pr}\).
More precisely,
\(\Cat{Pr}^{\kappa}\to\Cat{Pr}^{\lambda}\)
preserves \(\kappa\)-small limits for \(\lambda\geq\kappa>\aleph_{0}\).
This is reduced to the following two statements:

\begin{proposition}\label{prod}
  Let \(\kappa\) be a regular cardinal
  and \(I\) a \(\kappa\)-small set.
  Let \((\cat{C}_{i})_{i\in I}\)
  be an \(I\)-indexed family of \(\kappa\)-compactly generated categories.
  Then the product \(\prod_{i\in I}\cat{C}_{i}\)
  of large categories is
  a product in \(\Cat{Pr}^{\lambda}\)
  for \(\lambda\geq\kappa\).
\end{proposition}

\begin{proposition}\label{pb}
  Let \(\kappa\) be a regular cardinal
  and
  \begin{equation*}
    \begin{tikzcd}
      \cat{C}\ar[r]\ar[d]&
      \cat{D}\ar[d]\\
      \cat{C}'\ar[r]&
      \cat{D}'
    \end{tikzcd}
  \end{equation*}
  a cartesian diagram of large categories
  such that its subdiagram \(\cat{C}'\to\cat{D}'\gets\cat{D}\) is in~\(\Cat{Pr}^{\kappa}\).
  Then this is a cartesian diagram in~\(\Cat{Pr}^{\lambda}\)
  for \(\lambda\geq\max(\aleph_{1},\kappa)\).
\end{proposition}

\begin{remark}\label{xdjg80}
  The cardinality bounds in \cref{prod,pb}
  are optimal;
  see~\cite{Henry}.
\end{remark}

\begin{proof}[Proof of \cref{prod}]
  It suffices to show that
  the canonical functor
  \begin{equation*}
    \Ind_{\kappa}\biggl(\prod_{i\in I}(\cat{C}_{i})_{\kappa}\biggr)
    \to
    \prod_{i\in I}\cat{C}_{i}
  \end{equation*}
  is an equivalence.
  Assume that
  \(C_{i}\) is a \(\kappa\)-compact object of \(\cat{C}_{i}\).
  We prove that
  \((C_{i})_{i\in I}\) is \(\kappa\)-compact in \(\prod_{i\in I}\cat{C}_{i}\).
  Let \(P\) be a \(\kappa\)-directed poset
  and \(D\colon P\to\prod_{i\in I}\cat{C}_{i}\) be the diagram.
  Then we wish to show that the composite
  \begin{equation*}
    \injlim_{p\in P}\prod_{i\in I}
    \Map(C_{i},D(p)_{i})
    \to
    \prod_{i\in I}\injlim_{p\in P}
    \Map(C_{i},D(p)_{i})
    \to
    \prod_{i\in I}
    \Map\biggl(C_{i},\injlim_{p\in P}D(p)_{i}\biggr)
  \end{equation*}
  is an equivalence.
  The first map is an equivalence
  since \(\kappa\)-small limits and \(\kappa\)-filtered colimits
  commute in \(\Cat{Ani}\).
  The second map is an equivalence
  since \(C_{i}\) is \(\kappa\)-compact in \(\cat{C}_{i}\).

  We then show that the image generates the target category.
  Let \((C_{i})_{i\in I}\) be a general object
  of the target.
  Then \(C_{i}\)
  is the colimit of \(D_{i}\colon P_{i}\to\cat{C}_{i}\)
  such that \(P_{i}\) is a \(\kappa\)-directed poset.
  Then the original object is the colimit of
  \begin{equation*}
    D=
    \prod_{i\in I}D_{i}\colon
    \prod_{i\in I}P_{i}\to
    \prod_{i\in I}\cat{C}_{i}.
  \end{equation*}
  Note that \(\prod_{i\in I}P_{i}\) is \(\kappa\)-directed.
\end{proof}

\begin{proof}[Proof of \cref{pb}]
  We write
  \(\cat{C}_{0}=\cat{C}'_{\lambda}\times_{\cat{D}'_{\lambda}}\cat{D}_{\lambda}\).
  It suffices to show that
  \(\cat{C}_{0}\) generates
  \(\cat{C}\) under (\(\lambda\)-filtered) colimits.
  Let \(C\) be an arbitrary object.
  We form the diagram
  \begin{equation*}
    \begin{tikzcd}
      (\cat{C}_{\lambda})_{/C}\ar[r]\ar[d]&
      (\cat{D}_{\lambda})_{/D}\ar[d]\\
      (\cat{C}'_{\lambda})_{/C'}\ar[r]&
      (\cat{D}'_{\lambda})_{/D'}\rlap,
    \end{tikzcd}
  \end{equation*}
  which is again cartesian.
  It suffices to show that
  the top and left arrows are cofinal.
  This follows from~\cite[Lemma~5.4.5.6]{LurieHTT}.
\end{proof}

\subsection{Enriched \texorpdfstring{\(\Ind\)}{Ind}-categories}\label{ss:ecg}

We fix an object~\(\cat{V}\) of \(\CAlg(\Cat{Pr})\)\footnote{What we prove here might work in the associative case,
  but the commutative case is sufficient
  for our purposes.
} as the base of enrichment.
We take a regular cardinal~\(\kappa\)
such that \(\cat{V}\) belongs to \(\CAlg(\Cat{Pr}^{\kappa})\).
We use the machinery of enriched category theory
to study objects of \(\Mod_{\cat{V}}(\Cat{Pr}^{\lambda})\)
for \(\lambda\geq\kappa\).
Our reference for enriched category theory is~\cite{Heine23}.
The following is~\cite[Theorem~1.3]{Heine23}:

\begin{theorem}[Heine]\label{x1gpof}
  The symmetric monoidal category
  \(\Mod_{\cat{V}}(\Cat{Pr})\)
  is equivalent
  to \(\Cat{Pr}^{\cat{V}}\),
  the category
  of tensored \(\cat{V}\)-categories
  with presentable underlying categories
  and left adjoint \(\cat{V}\)-functors.
\end{theorem}

\begin{definition}\label{xwh2oi}
  Let \(\cat{A}\) be a \(\cat{V}\)-category
  admitting conical \(\lambda\)-small colimits and
  \(\lambda\)-compact tensors,
  i.e.,
  tensors
  with \(\lambda\)-compact objects in~\(\cat{V}\).
  Let \(\cat{C}\) be a presentable \(\cat{V}\)-category.
  We write
  \begin{equation*}
    \Fun^{\cat{V}}_{\lex{\lambda}}(\cat{A}^{\op},\cat{C})
    \subset
    \Fun^{\cat{V}}(\cat{A}^{\op},\cat{C})
    =
    \PShv^{\cat{V}}(\cat{A};\cat{C})
  \end{equation*}
  for the full \(\cat{V}\)-subcategory spanned by
  \(\cat{V}\)-functors
  preserving
  conical \(\lambda\)-small limits
  and \(\lambda\)-compact powers.
\end{definition}

\begin{example}\label{x2jvul}
  Representable functors
  belong to \(\Fun^{\cat{V}}_{\lex{\lambda}}(\cat{A}^{\op},\cat{V})\).
\end{example}

\begin{proposition}\label{x8ypbm}
  Let \(\cat{C}\) 
  be a presentable \(\cat{V}\)-category
  whose corresponding \(\cat{V}\)-module
  is in \(\Mod_{\cat{V}}(\Cat{Pr}^{\lambda})\).
  We write \(\cat{C}_{\lambda}\)
  for the full \(\cat{V}\)-subcategory
  spanned by \(\lambda\)-compact objects of
  its underlying category.
  Then \(\cat{C}\)
  is equivalent
  to \(\Fun^{\cat{V}}_{\lex{\lambda}}((\cat{C}_{\lambda})^{\op},\cat{V})\).
\end{proposition}

\begin{lemma}\label{xybys9}
Let \(\lambda\geq\kappa\) be a regular cardinal.
  In~\(\cat{V}\),
  (conical) \(\lambda\)-filtered colimits
  commute with (conical) \(\lambda\)-small limits
  and \(\lambda\)-compact powers.
\end{lemma}

\begin{proof}
The commutation with \(\lambda\)-small limits
  is a standard fact in unenriched category theory.
  We observe the commutation with \(\lambda\)-compact powers.
  Let \(V\) be a \(\lambda\)-compact object
  and \((W_{i})_{i}\) a \(\lambda\)-filtered diagram in~\(\cat{V}\).
  Consider a \(\kappa\)-compact object~\(V'\).
  The result follows from
  \begin{equation*}
    \begin{split}
      \Map\biggl(V',\injlim_{i}W_{i}^{V}\biggr)
      &\simeq
      \injlim_{i}\Map\bigl(V',W_{i}^{V}\bigr)
      \simeq
      \injlim_{i}\Map(V\otimes V',W_{i})\\
      &\simeq
      \Map\biggl(V\otimes V',\injlim_{i}W_{i}\biggr)
      \simeq
      \Map\biggl(V',\Bigl(\injlim_{i}W_{i}\Bigr)^{V}\biggr),
    \end{split}
  \end{equation*}
  where we have used
  the \(\lambda\)-compactness of \(V\otimes V'\).
\end{proof}

\begin{proof}[Proof of \cref{x8ypbm}]
  From the inclusion \(\cat{C}_{\lambda}\to\cat{C}\),
  we obtain
  a canonical morphism
  \begin{equation*}
    L\colon\PShv^{\cat{V}}(\cat{C}_{\lambda})\to\cat{C}
  \end{equation*}
  in~\(\Cat{Pr}^{\cat{V}}\)
  by~\cite[Theorem~5.1]{Heine23}
  (see also~\cite{Hinich23}).
  We wish to observe that it is a localization
  and determine the essential image of its right adjoint~\(R\).

Let \(C\) be an object of~\(\cat{C}\).
  We write it as a (conical) \(\lambda\)-filtered colimit
  \(\injlim_{i}C_{i}\)
  of \(\lambda\)-compact objects.
  Since \(L\) preserves \(\lambda\)-compact objects,
  \(R\) preserves \(\lambda\)-filtered colimits.
  This shows that
  \(L(R(C))\to C\) is an equivalence.
  Also, \cref{xybys9}
  implies that
  the image of~\(R\)
  is contained
  in \(\Fun^{\cat{V}}_{\lex{\lambda}}
  ((\cat{C}_{\lambda})^{\op},\cat{V})\).

  We then prove the reverse inclusion.
  We consider an object~\(F\)
  in \(\Fun^{\cat{V}}_{\lex{\lambda}}
  ((\cat{C}_{\lambda})^{\op},\cat{V})\).
  We then consider
  \(\cat{K}=(\cat{C}_{\lambda})_{/F}\)
  as an (unenriched) category,
  which is \(\lambda\)-filtered.
  Then we take the colimit~\(F'\) of the tautological diagram
  \(\cat{K}\to\cat{C}_{\lambda}\to\PShv^{\cat{V}}(\cat{C}_{\lambda})\)
  and obtain a morphism \(F'\to F\).
  Note that \(F'\) is inside the essential image of~\(R\).
  We claim that this is an equivalence.
  To see this,
  consider a \(\kappa\)-compact object~\(V\) of~\(\cat{V}\)
  and an object~\(C\) of~\(\cat{C}_{\lambda}\).
  Then we have
  \begin{equation*}
    \begin{split}
      \Map(V,F'(C))
      &\simeq\Map\biggl(V,\injlim_{C'}\cat{C}(C',C)\biggr)
      \simeq\injlim_{C'}\Map(V,\cat{C}(C',C))
      \simeq\injlim_{C'}\Map\bigl(\unit,\cat{C}\bigl(C',C^{V}\bigr)\bigr)\\
      &\simeq\Map\bigl(\unit,F\bigl(C^{V}\bigr)\bigr)
      \simeq\Map\bigl(\unit,(F(C))^{V}\bigr)
      \simeq\Map(V,F(C)),
    \end{split}
  \end{equation*}
  which shows the desired claim.
\end{proof}

\subsection{Limits in presentable \texorpdfstring{\(2\)}{2}-categories}\label{ss:fil}

Consider a presentable \(2\)-category~\(\cat{C}\).
By definition,
this is the image of \(\cat{C}(\kappa)\in2\Cat{Pr}^{\kappa}\)
for some regular cardinal~\(\kappa\).
For \(\lambda\geq\kappa\),
we consider \(\cat{C}(\lambda)=\Cat{Pr}^{\lambda}\otimes_{\Cat{Pr}^{\kappa}}\cat{C}(\kappa)\).
Then the underlying \(2\)-category
of~\(\cat{C}\)
is obtained as \(\injlim_{\kappa}\cat{C}(\kappa)\).
Each \(\cat{C}(\lambda)\) admits limits.
Thus,
to understand limits in~\(\cat{C}\),
it suffices to analyze
how the transition functors
\(\cat{C}(\lambda)\to\cat{C}(\mu)\)
for \(\lambda\leq\mu\) interact with limits.

\begin{theorem}\label{two_fil}
  Let \(\kappa\leq\lambda\) be regular cardinals
  and \(\cat{C}\) an object of \(2\Cat{Pr}^{\kappa}\).
  We consider the functor
  \(\cat{C}\to\Cat{Pr}^{\lambda}\otimes_{\Cat{Pr}^{\kappa}}\cat{C}\).
  \begin{enumerate}
    \item\label{i:lim}
      It preserves \(\kappa\)-small limits.
    \item\label{i:mono}
      It is a monomorphism of categories.
  \end{enumerate}
\end{theorem}

The consideration above implies the following:

\begin{corollary}\label{xyzmgb}
  Any presentable \(2\)-category admits small limits.
\end{corollary}

\begin{remark}\label{x4y2oo}
The essential content of~\cref{i:lim} of \cref{two_fil}
  is that
  it preserves pullbacks,
  since it is easy to see that
  it preserves \(\kappa\)-small products
  via ambidexterity.
\end{remark}

We proceed to prove \cref{two_fil}.
The key is the following:

\begin{lemma}\label{xpr2g6}
  Let \(\kappa\) be a regular cardinal.
  Let \(\cat{C}\) be an object of \(\CAlg(\Cat{Pr}^{\kappa})\)
  and \(\cat{C}\to\cat{D}\) a morphism in \(\CAlg(\Cat{Pr})\).
  Assume that
  \(\cat{C}\to\cat{D}\) preserves \(\kappa\)-small limits.
  Then the following hold
  for an object~\(\cat{M}\) of \(\Mod_{\cat{C}}(\Cat{Pr}^{\kappa})\):
  \begin{enumerate}
    \item\label{i:up_li}
      The functor
      \(\cat{M}\to\cat{D}\otimes_{\cat{C}}\cat{M}\)
      preserves \(\kappa\)-small limits.
    \item\label{i:up_mo}
      When \(\cat{C}\to\cat{D}\)
      is a monomorphism,
      the functor
      \(\cat{M}\to\cat{D}\otimes_{\cat{C}}\cat{M}\)
      is a monomorphism.
  \end{enumerate}
\end{lemma}

\begin{proof}
By \cref{x8ypbm},
  the functor \(\cat{M}\to\cat{D}\otimes_{\cat{C}}\cat{M}\)
  is identified with the functor
  \begin{equation*}
    \Fun_{\lex{\kappa}}^{\cat{C}}((\cat{M}_{\kappa})^{\op},\cat{C})
    \to
    \Fun_{\lex{\kappa}}^{\cat{C}}((\cat{M}_{\kappa})^{\op},\cat{D})
  \end{equation*}
  induced from \(\cat{C}\to\cat{D}\).
  From this, we see~\cref{i:up_li,i:up_mo}.
\end{proof}

\begin{proof}[Proof of \cref{two_fil}]
By \cref{xuizju,prod,pb},
  we can apply~\cref{i:up_li} of \cref{xpr2g6}
  to obtain~\cref{i:lim}.
  Furthermore,
  since \(\Cat{Pr}^{\kappa}\to\Cat{Pr}^{\lambda}\)
  is a monomorphism,
  we can apply~\cref{i:up_mo} of \cref{xpr2g6}
  to obtain~\cref{i:mono}.
\end{proof}

\subsection{The adjoint functor theorem}\label{ss:aft}

The same idea
as in the proof of \cref{xpr2g6} proves \cref{adjoint},
which states the existence of a right adjoint:

\begin{proof}[Proof of \cref{adjoint}]
  Let \(F\colon\cat{C}\to\cat{D}\)
  be a morphism in \(2\Cat{Pr}\).
  We choose a regular cardinal~\(\kappa\)
  and a morphism
  \(\cat{C}(\kappa)\to\cat{D}(\kappa)\)
  in \(2\Cat{Pr}^{\kappa}\)
  that induces~\(F\).
  For \(\lambda\geq\kappa\),
  we write
  \(\cat{C}(\lambda)\to\cat{D}(\lambda)\)
  for the induced morphism in \(2\Cat{Pr}^{\lambda}\).
  We need to see that
  \begin{equation*}
    \begin{tikzcd}
      \cat{C}(\lambda)\ar[r]\ar[d]&
      \cat{D}(\lambda)\ar[d]\\
      \cat{C}(\mu)\ar[r]&
      \cat{D}(\mu)
    \end{tikzcd}
  \end{equation*}
  is right adjointable
  for \(\kappa\leq\lambda\leq\mu\).
  We argue as in the proof of \cref{xpr2g6}.
  The diagram
  \begin{equation*}
    \begin{tikzcd}
      \Fun^{\Cat{Pr}^{\kappa}}_{\lex{\kappa}}((\cat{C}(\kappa)_{\kappa})^{\op},\Cat{Pr}^{\lambda})
      \ar[d]&
      \Fun^{\Cat{Pr}^{\kappa}}_{\lex{\kappa}}((\cat{D}(\kappa)_{\kappa})^{\op},\Cat{Pr}^{\lambda})
      \ar[l]\ar[d]\\
      \Fun^{\Cat{Pr}^{\kappa}}_{\lex{\kappa}}((\cat{C}(\kappa)_{\kappa})^{\op},\Cat{Pr}^{\mu})
      &
      \Fun^{\Cat{Pr}^{\kappa}}_{\lex{\kappa}}((\cat{D}(\kappa)_{\kappa})^{\op},\Cat{Pr}^{\mu})
      \ar[l]
    \end{tikzcd}
  \end{equation*}
  induced from \(\cat{C}(\kappa)_{\kappa}\to\cat{D}(\kappa)_{\kappa}\)
  and \(\Cat{Pr}^{\lambda}\to\Cat{Pr}^{\mu}\)
  commutes,
  which is the desired result.
\end{proof}

\section{Limits of presentable \texorpdfstring{\(2\)}{2}-categories}\label{s:san}

We prove \cref{false}
by showing that
\(2\Cat{Pr}\)
does not admit a certain pullback
in \cref{ss:culprit}.
Our argument relies on the notion of the complexity of a dominant morphism,
which we study in \cref{ss:cplx}.
We use this counterexample
to prove \cref{omega}
in \cref{ss:hom}.

\subsection{Closure under colimits}\label{ss:sub}

\begin{definition}\label{x7c5js}
  Let \(\kappa\) be a regular cardinal
  and \(\cat{A}\) an object of \(\Cat{Rex}^{\kappa}\),
  i.e., a category admitting \(\kappa\)-small colimits.
  For a replete\footnote{A full subcategory is called \emph{replete}
    if it is closed under equivalences.
  } full subcategory~\(\cat{S}\),
  we define another replete full subcategory \(T^{\alpha}_{\kappa}(\cat{S})\)
  for any ordinal~\(\alpha\).
  We first define \(T_{\kappa}(\cat{S})\)
  to be the full subcategory
  spanned by the pushouts
  of spans of the \(\kappa\)-small coproducts
  of objects in~\(\cat{S}\).
  We then define \(T_{\kappa}^{\alpha}(\cat{S})\)
  via transfinite recursion:
  \begin{equation*}
    T_{\kappa}^{\alpha}(\cat{S})=
    \begin{cases}
      \cat{S}&\text{if \(\alpha=0\),}\\
      T_{\kappa}(T_{\kappa}^{\alpha}(\cat{S}))&\text{if \(\alpha=\beta+1\),}\\
      \bigcup_{\beta<\alpha}T_{\kappa}^{\beta}(\cat{S})&\text{if \(\alpha\) is a limit ordinal.}
    \end{cases}
  \end{equation*}
  Note that \(\cat{A}\) is implicit in the notation.
\end{definition}

The following functoriality is clear:

\begin{lemma}\label{xup8l5}
  In the situation of \cref{x7c5js},
  we consider a morphism \(F\colon\cat{A}\to\cat{A}'\)
  in \(\Cat{Rex}^{\kappa}\)
  and write \(\cat{S}'\) for the essential image of~\(\cat{S}\).
  Then \(F\) restricts to 
  \(T_{\kappa}^{\alpha}(\cat{S})\to T_{\kappa}^{\alpha}(\cat{S}')\)
  for any ordinal~\(\alpha\).
\end{lemma}

Note that this process stabilizes:

\begin{lemma}\label{xxud2b}
  In the situation of \cref{x7c5js},
  for any ordinal \(\alpha\geq\kappa\),
  we have \(T_{\kappa}^{\kappa}(\cat{S})=T_{\kappa}^{\alpha}(\cat{S})\).
  In particular,
  \(T_{\kappa}^{\kappa}(\cat{S})\)
  is the smallest full subcategory
  of~\(\cat{A}\) that contains~\(\cat{S}\)
  and is closed under \(\kappa\)-small colimits.
\end{lemma}

\begin{proof}
  It suffices to show
  \(T_{\kappa}(T_{\kappa}^{\kappa}(\cat{S}))
  \subset T_{\kappa}^{\kappa}(\cat{S})\).
  Write an object of the left-hand side
  as
  \begin{equation*}
    \coprod_{i'\in I'}A'_{i'}
    \gets\coprod_{i\in I}A_{i}
    \to\coprod_{j\in J}B_{j},
  \end{equation*}
  where \(I\), \(I'\), and \(J\) are \(\kappa\)-small sets
  and \(A_{i}\), \(A'_{i'}\), and \(B_{j}\)
  are objects of \(T_{\kappa}^{\kappa}(\cat{S})\).
  By definition,
  we can take ordinals~\(\alpha_{i}\), \(\alpha'_{i'}\), and~\(\beta_{j}<\kappa\) 
  such that
  \(A_{i}\in T_{\kappa}^{\alpha_{i}}(\cat{S})\),
  \(A'_{i'}\in T_{\kappa}^{\alpha'_{i'}}(\cat{S})\),
  and \(B_{j}\in T_{\kappa}^{\beta_{j}}(\cat{S})\).
  Let \(\alpha\) be the supremum of these ordinals,
  which is still less than~\(\kappa\) by regularity.
  Then we have \(\alpha<\kappa\).
  This shows that this object is in
  \(T_{\kappa}^{\alpha+1}(\cat{S})\)
  and hence in \(T_{\kappa}^{\kappa}(\cat{S})\).
\end{proof}

\subsection{Dominant morphisms of presentable categories}\label{ss:surj}

We use the following
(possibly nonstandard) terminology:

\begin{definition}\label{xwx9ff}
  We call
  a morphism in~\(\Cat{Pr}\)
  \emph{dominant}
  if the image generates the target under colimits.
\end{definition}

\begin{example}\label{x9pjwc}
  For regular cardinals \(\kappa\leq\lambda\)
  and an object \(\cat{C}\in2\Cat{Pr}^{\kappa}\),
  the functor
  \(\cat{C}\to\Cat{Pr}^{\lambda}\otimes_{\Cat{Pr}^{\kappa}}\cat{C}\)
  is dominant.
  This comes from the fact
  that \(\Cat{Pr}^{\lambda}\) is generated by
  \(\Fun([1],\Cat{Ani})\),
  which lifts to \(\Cat{Pr}^{\kappa}\).
\end{example}

We prove some basic facts:

\begin{lemma}\label{xgl0iq}
  Any morphism
  \(\cat{C}\to\cat{E}\)
  in \(\Cat{Pr}\)
  factors uniquely
  as \(\cat{C}\to\cat{D}\to\cat{E}\)
  such that \(\cat{C}\to\cat{D}\) is dominant
  and \(\cat{D}\to\cat{E}\) is fully faithful.
\end{lemma}

\begin{proof}
  We choose~\(\kappa\)
  such that \(\cat{C}\to\cat{E}\) is in \(\Cat{Pr}^{\kappa}\).
  Then we can define~\(\cat{D}\)
  to be the full subcategory of~\(\cat{E}\)
  generated by the image of~\(\cat{C}_{\kappa}\).
  The uniqueness is clear.
\end{proof}

\begin{lemma}\label{xgd59c}
  Consider a morphism \(F\colon\cat{C}\to\cat{D}\) in \(\Cat{Pr}^{\kappa}\)
  for a regular cardinal~\(\kappa\).
  We write \(\cat{T}\) for the smallest full subcategory
  of~\(\cat{D}\) that contains \(F(\cat{C}_{\kappa})\)
  and is closed under \(\kappa\)-small colimits.
  Then \(F\) is dominant as a morphism in~\(\Cat{Pr}\)
  if and only if
  any object of~\(\cat{D}_{\kappa}\)
  is a retract of an object of~\(\cat{T}\).
\end{lemma}

\begin{proof}
Since the ``if'' direction is clear,
  we prove the ``only if'' direction.
  We obtain a fully faithful functor \(\Ind_{\kappa}(\cat{T})\to\cat{D}\)
  and the result follows from \cref{xgl0iq}.
\end{proof}

From now on,
we specialize to the stable situation.
The following is fundamental
and used implicitly below:

\begin{lemma}\label{x6g4gr}
  A morphism in~\(\Cat{Pr}_{\st}\)
  is dominant if and only if
  its cofiber is zero.
\end{lemma}

\begin{proof}
  The “only if” direction is clear.
  To prove the “if” direction,
  we use \cref{xgl0iq}.
  So we have to show that
  if a fully faithful functor
  \(F\colon\cat{C}\to\cat{D}\) in~\(\Cat{Pr}_{\st}\)
  has zero cofiber,
  it is an equivalence.
  Suppose not and let \(D\) be an object of~\(\cat{D}\)
  such that \(F(G(D))\to D\)
  is not an equivalence,
  where \(G\) denotes the right adjoint of~\(F\).
  Then the cofiber of this map
  is killed by~\(G\), which is a contradiction.
\end{proof}

\subsection{Complexity of a dominant morphism}\label{ss:cplx}

We here introduce the notion of complexity
that measures how far a dominant morphism is from
being essentially surjective.

\begin{definition}\label{xohpra}
  Let \(\kappa>\aleph_{0}\) be a regular cardinal
  and \(F\colon\cat{C}\to\cat{D}\) a morphism in \(\Cat{Pr}_{\st}^{\kappa}\)
  whose cofiber is zero.
  We write \(c_{\kappa}(F)\)
  for the smallest ordinal~\(\alpha\)
  such that \(T_{\kappa}^{\alpha}(F(\cat{C}_{\kappa}))=\cat{D}_{\kappa}\)
  (see \cref{x7c5js} for the notation),
  which exists by \cref{xgd59c,xxud2b}.
\end{definition}

We then consider the relation of~\(c_{\kappa}(F)\)
as \(\kappa\) varies:

\begin{proposition}\label{x48xj0}
  Let \(\aleph_{0}<\kappa\leq\lambda\) be regular cardinals.
  For a morphism \(\cat{C}\to\cat{D}\) in \(\Cat{Pr}^{\kappa}_{\st}\)
  with zero cofiber,
  we have
  \(c_{\lambda}(F)\leq c_{\kappa}(F)+\omega+2\).
\end{proposition}

We need the following to prove this:

\begin{lemma}\label{xp5oom}
  Let \(\kappa\leq\lambda\) be regular cardinals.
  Let \(\kappa\) be a regular cardinal
  and \(\cat{A}\) be a small \(\infty\)-category with \(\kappa\)-small colimits.
  Then any \(\lambda\)-compact object of \(\cat{C}=\Ind_{\kappa}(\cat{A})\)
  for a regular cardinal~\(\lambda\)
  is a retract of a \(\lambda\)-small colimits
  of a diagram in~\(\cat{A}\).
\end{lemma}

\begin{proof}
  Let \(C\) be a \(\lambda\)-compact object.
  We take a (\(\kappa\)-directed) diagram \(F\colon P\to\cat{A}\)
  whose colimit in~\(\cat{C}\) is~\(C\).
  By the proof of~\cite[Corollary~4.2.3.11]{LurieHTT},
  we see that it is a \(\lambda\)-filtered colimit
  of the colimits of \(\lambda\)-small diagrams in~\(\cat{A}\).
  Since \(C\) is \(\lambda\)-compact,
  it must be a retract of one of them.
\end{proof}

\begin{remark}\label{xbcsam}
  In the situation of \cref{xp5oom},
  we cannot arrange the diagram
  to be \(\kappa\)-filtered.
Indeed, one can observe that
  \(\SS^{\oplus\aleph_{\omega}}\) is not
  a \(\aleph_{\omega+1}\)-small \(\aleph_{1}\)-filtered colimit
  of \(\aleph_{1}\)-compact spectra.
\end{remark}

\begin{proof}[Proof of \cref{x48xj0}]
It suffices to prove
  \(T_{\lambda}^{\omega+2}(\cat{D}_{\kappa})=\cat{D}_{\lambda}\).
  Consider an object \(D\in\cat{D}_{\lambda}\).
  By \cref{xp5oom},
  we can write it as a retract
  of a \(\lambda\)-small colimit of a diagram in~\(\cat{D}_{\kappa}\).
  Therefore,
  we need to show that
  \(D\in T_{\lambda}^{\omega+1}(\cat{D}_{\kappa})\)
  when \(D\) is a \(\lambda\)-small colimit of objects in~\(\cat{D}_{\kappa}\).

  Any such \(\lambda\)-small colimit~\(D\)
  is the geometric realization
  of some simplicial object~\(D_{\bullet}\)
  whose terms are \(\lambda\)-small coproducts
  of objects in~\(\cat{D}_{\kappa}\).
  We forget face maps and consider it as a semisimplicial object.\footnote{This matters,
    since the partial geometric realization
    changes when we forget face maps.
  }
  Then we write~\(D\) as the partial geometric realization
  \(\injlim_{n}E_{n}\).
  We then obtain \(E_{0}=D_{0}\),
  the cofiber sequences
  \(\Sigma^{n}D_{n+1}\to E_{n}\to E_{n+1}\)
  for any \(n\geq 0\),
  and
  the cofiber sequence
  \begin{equation*}
    \bigoplus_{n\geq0}E_{n}
    \to
    \bigoplus_{n\geq0}E_{n}
    \to
    D.
  \end{equation*}
  From this,
  we have \(E_{n}\in T_{\kappa}^{n}(\cat{D}_{\kappa})\)
  for \(n\geq1\) by induction,
  and hence \(D\in T_{\kappa}^{\omega+1}(\cat{D}_{\kappa})\) holds.
  This completes the proof.
\end{proof}

We here see an example of a dominant morphism
with large complexity:

\begin{example}\label{x1814s}
  Let \(\kappa\) be a regular cardinal.
  We construct a morphism
  \begin{equation*}
    F\colon
    \cat{C}=\Fun(\kappa^{\delta},\Cat{Sp})
    \to
    \Fun(\kappa^{\op},\Cat{Sp})=\cat{D},
  \end{equation*}
  in~\(\Cat{Pr}_{\st}\)
  where \(\kappa^{\delta}\) is the discrete category
  with the underlying set~\(\kappa\).
  For \(\alpha\in\kappa\),
  we write~\(\SS(\alpha)\) and~\(Y(\alpha)\)
  for the images of~\(\alpha\) under
  the Yoneda embeddings from~\(\kappa^{\delta}\) and~\(\kappa\)
  to~\(\cat{C}\) and~\(\cat{D}\), respectively.
  We define~\(F\) to be the one
  that maps~\(\SS(\alpha)\) to
  \begin{equation*}
    X(\alpha)
    =
    \cofib\biggl(\injlim_{\beta<\alpha}Y(\beta)\to Y(\alpha)\biggr).
  \end{equation*}
  Since the index is \(\kappa\)-small,
  this is a \(\kappa\)-compact object.
  Therefore, \(F\) is a morphism in~\(\Cat{Pr}_{\st}^{\kappa}\).
  We fix an uncountable regular cardinal \(\lambda\geq\kappa\)
  and show \(c_{\lambda}(F)\geq\kappa\).

  For \(D\colon\kappa^{\op}\to\Cat{Sp}\)
  and an ordinal~\(\alpha\),
  we say that it is \emph{\(\alpha\)-nilpotent}
  if \(D(\beta+\alpha)\to D(\beta)\)
  is zero for any ordinal~\(\beta\)
  with \(\beta+\alpha\in\kappa\).
  For example,
  \(X(\alpha)\) is \(1\)-nilpotent for any~\(\alpha\).

  We claim that
  for any ordinal~\(\alpha\),
  any object in
  \(T^{\alpha}_{\lambda}(F(\cat{C}_{\lambda}))\)
  is \(2^{\alpha}\)-nilpotent.
  We prove this by transfinite induction on~\(\alpha\).
  This is clear when \(\alpha=0\).
  To prove this,
  it suffices to show that
  if \(D\) and \(D'\) are \(\alpha\)-nilpotent,
  then the cofiber~\(D''\) of any map \(D\to D'\)
  is (\(\alpha\cdot2\))-nilpotent;
  this follows from observing
  that the map \(D''(\beta)\gets D''(\beta+\alpha\cdot2)\)
  factors through zero,
  which is the cofiber
  of zeros
  between \(D(\beta+\alpha)\gets D(\beta+\alpha\cdot2)\)
  and \(D'(\beta)\gets D'(\beta+\alpha)\).

  Now, the constant diagram~\(Y(0)\)
  is not \(\alpha\)-nilpotent
  for any~\(\alpha\in\kappa\).
Therefore,
  since \(\lvert 2^{\alpha}\rvert=\lvert\alpha\rvert\)
  for any infinite ordinal~\(\alpha\),
  the constant diagram~\(Y(0)\)
  does not belong to
  \(T_{\lambda}^{\kappa}(F(\cat{C}_{\lambda}))\).
  This shows \(c_{\lambda}(F)\geq\kappa\).
\end{example}

\subsection{A nonpresentable pullback}\label{ss:culprit}

We prove the following result,
which implies \cref{false}:

\begin{theorem}\label{x90m7h}
  We consider
  \({\cofib}\colon\Fun([1],\Cat{Pr}_{\st})\to\Cat{Pr}_{\st}\).
  Then
  \begin{equation*}
    \ker(\cofib)=\Fun([1],\Cat{Pr}_{\st})\times_{\Cat{Pr}_{\st}}0
  \end{equation*}
  does not exist in \(2\Cat{Pr}\).
\end{theorem}

We first see
the following evidence:

\begin{definition}\label{x2xh6s}
  For a regular cardinal \(\kappa>\aleph_{0}\),
  we write
  \begin{equation*}
    \Cat{Dom}^{\kappa}
    =\ker({\cofib}\colon\Fun([1],\Cat{Pr}^{\kappa}_{\st})
    \to
    \Cat{Pr}^{\kappa}_{\st}),
  \end{equation*}
  which we regard as an object of \(2\Cat{Pr}^{\kappa}\)
  by \cref{pb}.
\end{definition}

\begin{proposition}\label{xgfihx}
  For regular cardinals~\(\mu\geq\lambda>\kappa>\aleph_{0}\),
  the functor
  \begin{equation*}
    \Cat{Pr}^{\mu}\otimes_{\Cat{Pr}^{\kappa}}\Cat{Dom}^{\kappa}
    \to\Cat{Pr}^{\mu}\otimes_{\Cat{Pr}^{\lambda}}\Cat{Dom}^{\lambda}
  \end{equation*}
  is not dominant.
  In particular,
  \(2\Cat{Pr}^{\kappa}\to2\Cat{Pr}^{\lambda}\)
  does not preserve pullbacks.
\end{proposition}

\begin{proof}
  We prove
  that
  the images of
  \(\Cat{Dom}^{\kappa}\to\Cat{Dom}^{\mu}\)
  and \(\Cat{Dom}^{\lambda}\to\Cat{Dom}^{\mu}\)
  generate different categories under colimits.
  By \cref{x9pjwc}, this suffices.

  Note the ordinal addition
  \(\lambda+\lambda=\lambda\).
  We consider the commutative square
  \begin{equation*}
    \begin{tikzcd}[column sep=large]
      \Fun(\lambda^{\delta},\Cat{Sp})\oplus\Fun(\lambda^{\delta},\Cat{Sp})
      \ar[r,"X\oplus X"]\ar[d,"\simeq"']&
      \Fun(\lambda^{\op},\Cat{Sp})\oplus\Fun(\lambda^{\op},\Cat{Sp})\ar[d]\\
      \Fun((\lambda+\lambda)^{\delta},\Cat{Sp})\ar[r,"X"]&
      \Fun((\lambda+\lambda)^{\op},\Cat{Sp})
    \end{tikzcd}
  \end{equation*}
  in \(\Cat{Pr}_{\st}\),
  where \(X\) is as in \cref{x1814s}
  (see~\cref{xxf9hh} for the right vertical arrow).
  We regard this as a monomorphism
  \(X\oplus X\to X\) in \(\Cat{Dom}^{\lambda}\).
  We conclude the proof by showing
  that for any object~\(F\colon\cat{C}\to\cat{D}\)
  in \(\Cat{Dom}^{\kappa}\),
  \begin{equation*}
    \Map_{\Cat{Dom}^{\mu}}(F,X\oplus X)
    \to
    \Map_{\Cat{Dom}^{\mu}}(F,X)
  \end{equation*}
  is an equivalence.

  Now,
  in general,
  for a morphism \(F'\colon\cat{C}'\to\cat{D}'\) in \(\Cat{Dom}^{\mu}\),
  we have a factorization
  \begin{equation*}
    \begin{tikzcd}
      \cat{C}_{\mu}\ar[r]\ar[d]&
      \cat{D}_{\mu}\ar[d]&
      {}\\
      \cat{C}'_{\mu}\ar[r]&
      T_{\mu}^{\kappa+\omega+2}(F'(\cat{C}'_{\mu}))\ar[r]&
      \cat{D}'_{\mu}
    \end{tikzcd}
  \end{equation*}
  by \cref{xxud2b,x48xj0}.
  Therefore,
  it suffices to show that
  the induced functor
  \begin{equation*}
    T_{\mu}^{\alpha}
    (
    (X\oplus X)
    ((\Fun(\lambda^{\delta},\Cat{Sp})\oplus\Fun(\lambda^{\delta},\Cat{Sp}))_{\mu})
    )
    \to
    T_{\mu}^{\alpha}
    (
    X(\Fun((\lambda+\lambda)^{\delta},\Cat{Sp})_{\mu})
    )
  \end{equation*}
  is an equivalence
  for any ordinal~\(\alpha\leq\kappa+\omega+2\).
  By \cref{x1814s},
  we see that objects on both sides
  are \(2^{\alpha}\)-nilpotent.
  The desired result follows from \cref{xxf9hh} below
  and transfinite induction on~\(\alpha\).
\end{proof}

\begin{lemma}\label{xxf9hh}
  Let \(\kappa\) be an infinite cardinal regarded as an ordinal.
  Let \(\gamma\) be an ordinal.
  We consider the left and right Kan extension functors
  \begin{align*}
    j_{!}&\colon\Fun(\kappa^{\op},\Cat{Sp})\to\Fun((\kappa+\gamma)^{\op},\Cat{Sp}),&
    i_{*}&\colon\Fun(\gamma^{\op},\Cat{Sp})\to\Fun((\kappa+\gamma)^{\op},\Cat{Sp})
  \end{align*}
  to obtain the functor
  \begin{equation*}
    \Fun(\kappa^{\op},\Cat{Sp})\oplus\Fun(\gamma^{\op},\Cat{Sp})
    \to\Fun((\kappa+\gamma)^{\op},\Cat{Sp}).
  \end{equation*}
  We consider an ordinal \(\alpha\in\kappa\).
  This functor restricts to an equivalence
  on the full subcategories
  of \(\alpha\)-nilpotent objects (see \cref{x1814s}).
\end{lemma}

\begin{proof}
  We write \(j_{*}\)
  for the right Kan extension functor
  \(\Fun(\kappa^{\op},\Cat{Sp})\to\Fun((\kappa+\gamma)^{\op},\Cat{Sp})\).
  The key observation is that
  when \(D\colon\kappa^{\op}\to\Cat{Sp}\)
  is \(\alpha\)-nilpotent,
  then the canonical morphism \(j_{!}(D)\to j_{*}(D)\)
  is an equivalence.

  We show its essential surjectivity.
  Consider
  \(D\colon(\kappa+\gamma)^{\op}\to\Cat{Sp}\).
  We have a cofiber sequence
  \begin{equation*}
    j_{!}(D\rvert_{\kappa^{\op}})
    \to
    D
    \to
    i_{*}(D\rvert_{\gamma^{\op}}).
  \end{equation*}
We must show that this splits
  when \(D\) is \(\alpha\)-nilpotent.
  To see this,
  we use the canonical
  morphism \(D\to j_{*}(D\rvert_{\kappa^{\op}})\).

We then prove full faithfulness.
  This reduces to showing
  that any morphism \(i_{*}(E)\to j_{!}(D)\) is zero
  when \(D\colon\kappa^{\op}\to\Cat{Sp}\)
  and \(E\colon\gamma^{\op}\to\Cat{Sp}\)
  are \(\alpha\)-nilpotent.
  The desired claim follows from the observation
  that
  \(i_{*}(E)\to j_{*}(D)\) is always zero.
\end{proof}

\begin{proof}[Proof of \cref{x90m7h}]
  Suppose that the limit \((\cat{L}(\kappa))_{\kappa}\) exists.
  We have a morphism
  \(\cat{L}(\kappa)\to\Cat{Dom}^{\kappa}\)
  for any~\(\kappa\gg0\).
  We also have
  \(\Cat{Pr}^{\lambda}\otimes_{\Cat{Pr}^{\kappa}}\Cat{Dom}^{\kappa}
  \to\cat{L}(\lambda)\)
  for \(\lambda\gg\kappa\)\footnote{This symbol~\(\gg\) here is used as “sufficiently large”
    as usual,
    and should not be confused
    with the use in~\cite[Definition~A.2.6.3]{LurieHTT}.
  }.
  Hence,
  repeating this,
  we can construct a sequence of regular cardinals
  \(\kappa_{0}\leq\kappa_{1}\leq\dotsb\)
  and morphisms
  \begin{equation*}
    \Cat{Dom}^{\kappa_{0}}
    \to\cat{L}(\kappa_{1})\to\Cat{Dom}^{\kappa_{1}}
    \to\cat{L}(\kappa_{2})\to\dotsb
  \end{equation*}
  with the following properties
  for any~\(n\geq0\):
  \begin{itemize}
    \item
      The morphism
      \(\Cat{Dom}^{\kappa_{n}}\to\cat{L}(\kappa_{n+1})\)
      is in \(\Mod_{\Cat{Pr}^{\kappa_{n}}}(\Cat{Pr}^{\kappa_{n+1}})\)
    \item
      The morphism
      \(\cat{L}(\kappa_{n+1})\to\Cat{Dom}^{\kappa_{n+1}}\)
      is in \(2\Cat{Pr}^{\kappa_{n+1}}\).
    \item
      The composite
      \(\Cat{Dom}^{\kappa_{n}}\to\Cat{Dom}^{\kappa_{n+1}}\)
      is the tautological morphism.
    \item
      The composite
      \(\cat{L}(\kappa_{n+1})\to\cat{L}(\kappa_{n+2})\)
      is the structure morphism.
  \end{itemize}
  We then set \(\kappa_{\infty}=(\sup_{n}\kappa_{n})^{+}\) and
  similarly construct a sequence of regular cardinals
  \(\kappa_{\infty}=\lambda_{0}\leq\lambda_{1}\leq\dotsb\)
  and morphisms
  \begin{equation*}
    \Cat{Dom}^{\lambda_{0}}
    \to\cat{L}(\lambda_{1})\to\Cat{Dom}^{\lambda_{1}}
    \to\cat{L}(\lambda_{2})\to\dotsb
  \end{equation*}
  satisfying the same properties.
  We set \(\lambda_{\infty}=(\sup_{n}\lambda_{n})^{+}\).
  Then the composite
  \begin{equation*}
    \injlim_{n}
    \Cat{Pr}^{\lambda_{\infty}}
    \otimes_{\Cat{Pr}^{\kappa_{n}}}
    \Cat{Dom}^{\kappa_{n}}
    \to
    \injlim_{n}
    \Cat{Pr}^{\lambda_{\infty}}
    \otimes_{\Cat{Pr}^{\lambda_{n}}}
    \Cat{Dom}^{\lambda_{n}}
  \end{equation*}
  in \(2\Cat{Pr}^{\lambda_{\infty}}\)
  is an equivalence.
  This implies that
  \begin{equation*}
    \Cat{Pr}^{\lambda_{\infty}}
    \otimes_{\Cat{Pr}^{\kappa_{\infty}}}
    \Cat{Dom}^{\kappa_{\infty}}
    \to
    \injlim_{n}
    \Cat{Pr}^{\lambda_{\infty}}
    \otimes_{\Cat{Pr}^{\lambda_{n}}}
    \Cat{Dom}^{\lambda_{n}}
  \end{equation*}
  is essentially surjective.
  However,
  the closure under colimits
  of their respective essential images
  in \(\Cat{Dom}^{\lambda_{\infty}}\)
  are different,
  as the proof of \cref{xgfihx} shows.
\end{proof}

\subsection{A nonpresentable mapping category}\label{ss:hom}

We conclude this paper by deducing \cref{omega}.

\begin{lemma}\label{xswc9r}
  There is a diagram \(\cat{C}'\to\cat{D}'\gets\cat{D}\)
  in \(\CAlg(2\Cat{Pr})\)
  that does not admit a limit.
\end{lemma}

\begin{proof}
  We apply the free commutative algebra functor
  \(2\Cat{Pr}\to\CAlg(2\Cat{Pr})\)
  to the diagram in \cref{x90m7h}.
  Concretely,
  it sends~\(\cat{C}\) to
  \(\bigoplus_{n\geq0}\cat{C}^{\otimes n}_{h\Sigma_{n}}\),
  which is also a product by ambidexterity.
  The pullback must also be graded,
  and by inspecting the graded piece in degree~\(1\),
  the result follows.
\end{proof}

\begin{lemma}\label{xq124c}
For a diagram \(\cat{C}'\to\cat{D}'\gets\cat{D}\)
  in \(\CAlg(2\Cat{Pr})\),
  the limit
  \begin{equation*}
    \begin{tikzcd}
      \cat{C}\ar[r]\ar[d]&
      \Mod_{\Mod_{\cat{C}'}(2\Cat{Pr})}(3\Cat{Pr})\ar[d]\\
      \Mod_{\Mod_{\cat{D}}(2\Cat{Pr})}(3\Cat{Pr})\ar[r]&
      \Mod_{\Mod_{\cat{D}'}(2\Cat{Pr})}(3\Cat{Pr})
    \end{tikzcd}
  \end{equation*}
  exists in \(\CAlg(4\Cat{Pr})\).
\end{lemma}

\begin{proof}
  We can apply~\cite[Proposition~5.5.10]{Stefanich},
  since each arrow has a left adjoint in~\(4\Cat{Pr}\)
  by ambidexterity.
\end{proof}

\begin{proof}[Proof of \cref{omega}]
  This follows from combining \cref{xswc9r,xq124c}:
  If \(\End_{\cat{C}}(\unit)\) does not exist presentably,
  \(\cat{C}\) is such an example.
  Otherwise,
  \(\End_{\cat{C}}(\unit)\) is such an example,
  since \(\End_{\End_{\cat{C}}(\unit)}(\unit)\) does not exist presentably.
\end{proof}

\bibliographystyle{plain}
\let\SS\oldSS  \newcommand{\yyyy}[1]{}

\end{document}